\documentclass[11pt]{article}

\usepackage[parfill]{parskip}
\usepackage{color}
\usepackage[utf8]{inputenc} 
\usepackage[T1]{fontenc}    
\usepackage{hyperref}       
\usepackage{url}            
\usepackage{booktabs}       
\usepackage{amsfonts}       
\usepackage{nicefrac}       
\usepackage{microtype}      

\usepackage{subcaption}
\usepackage{algorithm}%
\usepackage{algpseudocode}
\usepackage{braket}
\usepackage{graphicx}
\usepackage{framed}
\usepackage{amsfonts}
\usepackage{mathtools}
\usepackage{fullpage}
\usepackage{amsmath}
\usepackage{amsthm}

\newcommand{\norm}[1]{\left\lVert#1\right\rVert}

\usepackage[colorinlistoftodos,bordercolor=orange,backgroundcolor=orange!20,linecolor=orange,textsize=scriptsize]{todonotes}

\newtheorem{theorem}{Theorem}

\newtheorem{lemma}{Lemma}
\newtheorem{definition}{Definition}
\newtheorem{remark}{Remark}
\newtheorem{proposition}{Proposition}
\newtheorem{assumption}{Assumption}

\newcommand{\R}{\mathbb{R}}

\newcommand{\bI}{\mathbf{I}}
\newcommand{\bM}{\mathbf{M}}
\newcommand{\bA}{\mathbf{A}}

\newcommand{\bG}{\mathbf{G}}
\newcommand{\bD}{\mathbf{D}}
\newcommand{\bX}{\mathbf{X}}
\newcommand{\bY}{\mathbf{Y}}
\newcommand{\bQ}{\mathbf{Q}}

\newcommand{\eqdef}{\overset{\text{def}}{=}}

\makeatletter
\renewcommand{\algorithmicrequire}{\textbf{Parameters:}}
\renewcommand{\algorithmicensure}{\textbf{Initialization:}}
\makeatother

\begin{document}

		\title{\bf Parallel Stochastic Newton Method}
		


	\author{Mojm\'{i}r   Mutn\'{y}	\\
		ETH Z\"urich\\
		Switzerland\\		
		\and
		Peter Richt\'{a}rik \\ University of Edinburgh \\ United Kingdom }
	
	\maketitle
	
	\begin{abstract}
	We propose a parallel stochastic Newton method (PSN) for minimizing unconstrained smooth convex functions. We analyze the method in the strongly convex case, and give conditions under which acceleration can be expected when compared to its serial counterpart. We show how PSN can be applied to the empirical risk minimization problem, and demonstrate the practical efficiency of the method through numerical experiments and models of simple matrix classes. 
	
		
	\end{abstract}
	\bibliographystyle{plain}
	
%
%
%
	{\footnotesize
		\tableofcontents
	}

	\newpage

	\section{Introduction}
	This work presents a novel parallel algorithm for minimizing an unconstrained strongly convex function.  This work is motivated by the possibility of better leveraging the structure in surrogate approximation, and the need for efficient optimization methods of high dimensional functions. The age of ``Big Data'' demands efficient algorithms to solve optimization problems that arise, for example, in fitting of large statistical models or large systems of equations. These new demands define open questions in algorithm design that make previously efficient algorithms obsolete. 
	
	For example, in this context, classical second order methods such as Newton method are not applicable as the inversion step of the algorithm is too costly ($O(n^3)$) to be performed in big data settings. Due to this reason, first-order algorithms enjoy huge popularity in the field of practicing optimizers, mainly in the field of machine learning. Recent years have shown that randomization and use of second-order information can lead to better convergence properties of algorithms. A prime example of this utilization are coordinate methods; to mention a few: \cite{NSYNC,PCDM,QUARTZ,APPROX}. Another school, more traditionally grouped under term second-order, has seen a plethora of algorithms in recent year with modified LBFGs \cite{Liu1989,Gower2016} methods to sub-sampled Newton methods \cite{Pilanci2015,Roosta-Khorasani2016,Roosta-Khorasani2016a,Erdogdu2015}, which coincide with the direction of this work. 
	
	In the current trend, computations are increasingly becoming parallelized, and the increase in performance is usually achieved by including more computing units solving a problem in parallel. Such architectures demand an efficient design of parallel algorithms that are able to exploit the parallel nature of computing clusters. An effort has been undertaken to provide theoretical certificates on convergence of parallel optimization algorithms, to name a few, \cite{PCDM,HYDRA}, or from class of stochastic methods \cite{Zinkevich2010,Recht2011}.
	
	We chose to extend an existing algorithm that utilizes curvature information, called SDNA \cite{SDNA}, which improves on standard coordinate methods such as SDCA \cite{Shalev-Shwartz2013} (of which parallel versions exist \cite{PCDM},\cite{Richtarik2012}), and present theoretical certificates on parallelization efficiency of this algorithm along with analysis of special matrix classes. These analyses hint to better theoretical and practical than parallel coordinate descent method (PCDM) \cite{PCDM}. 
	
	Further, we focus on big data application with statistical model of learning known as Empirical Risk Minimization (ERM)
	\begin{equation}
		\min_{w \in \mathbb{R}^d} \left[P(w)  = \frac{1}{n}\sum_{i=1}^{n} \phi_i(a_i^\top w )+ \lambda g(w) \right],
	\end{equation}
	which fits many of the statistical estimation models such as Ridge Regression. We present modified PSN for this type of problems that truck dual and primal variables. 
	
	\subsection{Contributions}
	The main contribution of this paper is the {\em design of a novel parallel algorithm} and its subsequent {\em novel theoretical analysis}. In the case of a smooth objective function, we present convergence analysis with proofs. The method in its simple serial case reduces to variants of algorithms introduced in \cite{SDNA} or \cite{Shalev-Shwartz2013}. 
	
	We identify parameters of the problem that determine its parallelizability and analyze them in special cases. To do this, we generalize two classes of quadratic optimization problems parametrized by one parameter and analytically calculate the convergence rates for them. 
	
	This work utilizes the research on sampling analyzed in paper \cite{ESO}, and is contrasted mainly with another parallel algorithm - parallel coordinate method (PCDM) analyzed in \cite{NSYNC}. Furthermore, it generalizes further the class of coordinate methods beyond the generalization of blocks. In this work, the sampled blocks of the over-approximation are not fixed and can overlap. The choice of sampling leading to non-overlapping and fixed blocks has been analyzed previously in \cite{Fountoulakis2015,Marecek2014} and mainly in \cite{Richtarik2012}.
	
	\subsection{Notation}
	
		\paragraph{Vectors.}
		In this work, we use the convention that vectors in $\mathbb{R}^n$ are labeled with lowercase Latin letters.	By $e_1, e_2, \dots e_n$ we denote the standard basis vectors in $\mathbb{R}^n$. The $i$th element of a vector $x \in \mathbb{R}^n$ therefore is $x_i= e_i^\top x$.
		The standard Euclidean inner product between vectors in $\mathbb{R}^n$ is given by
		$\braket{x,y} \eqdef x^\top y=\sum_{i=1}^{n}x_iy_i.$
	
		\paragraph{Matrices.} 
		We use the convention that matrices in $\mathbb{R}^{n\times n}$ are labeled with uppercase bold Latin letters.	By $\mathbf{I}$ we denote the identity matrix in $\mathbb{R}^{n\times n}$. The diagonal matrix with vector $w \in \mathbb{R}^n$  on the diagonal is denoted by $\mathbf{D}(w)$. We write $\mathbf{M} \succeq 0$ (resp.\ $\mathbf{M} \succ 0$) to indicate that $\mathbf{M}$ is symmetric positive semi-definite (resp.\ symmetric positive definite). Elements of a matrix $\mathbf{A} \in \mathbb{R}^{n\times n}$ are denoted in the natural way: $\mathbf{A}_{ij}\eqdef  e_i^\top \mathbf{A}e_j$.

		\paragraph{Sampling a Matrix.} 
		Let $S$ be a non-empty subset of  $[n]:=\{1,2, \dots, n\}$.
		We let $\bI_{:S}$ be the $n \times |S|$ matrix  composed of columns $i \in S$ of the $n\times n$ identity matrix $\bI$. Note that $\bI_{:S}^\top \bI_{:S}$ is the $|S|\times |S|$ identity matrix.

		 Given an invertible matrix $\bM\in \R^{n\times n}$, we can extract its principal  $|S| \times |S|$ sub-matrix  corresponding to the rows and columns indexed by $S$ by
			\begin{equation} \label{eq:smaller_splice}
				\mathbf{M}_{SS} \eqdef \bI_{:S}^\top \mathbf{M} \bI_{:S}.
			\end{equation}	
		It will be also convenient to  define

		\begin{equation}\label{eq:slice}
		\bM_S \eqdef \bI_{:S}\bM_{SS} \bI_{:S}^\top
		\end{equation}
		 and 
		 \begin{equation} \label{eq:inverse}
		 (\bM_{S})^{-1} \eqdef \bI_{:S}(\bM_{SS})^{-1} \bI_{:S}^\top
		 \end{equation} as we shall use these matrices often. Notice that $\mathbf{M}_{S}$ is the $n\times n$ matrix obtained from $\bM$ by retaining elements $\bM_{ij}$ for $i \in S$ and $j \in S$; and all the other elements  set to zero. On the other hand, $(\bM_{S})^{-1}$ is obtained from $\bM$ by zeroing out the elements corresponding to $i,j\notin S$ and inverting, ''in place'', the $|S|\times |S|$ matrix composed of elements $i,j\in S$. 
		
			Additionally, for any vector $h \in \mathbb{R}^n$ and $\emptyset \neq S\subseteq [n]$ we define $h_S\in \R^n$ by
			\begin{equation}
				h_S \eqdef \mathbf{I}_{:S}h=\sum_{i \in S}h_ie_i.
			\end{equation}
That is, $h_S$ is obtained from 	$h$ by zeroing out elements $i\notin S$.
		
	\subsection{Randomly sampled sub-matrices}
		In this section we inroduce some basic  notation which will be needed throughot the paper, following the convention established  in \cite{ESO}. 
	
	A {\em sampling}, denoted $\hat{S}$, is a random set-valued mapping with values being subsets of $[n]:= 
		\{1, \dots, n\}$.  With each sampling we associate a {\em probability matrix}, $\mathbf{P} = \mathbf{P}(\hat{S})$, defined via
		\begin{equation}
		\mathbf{P}_{ij}\eqdef\mathbb{P}(i \in \hat{S} \text{ and } j \in \hat{S}), \qquad i,j\in [n].
		\end{equation} 
		We drop the index $\hat{S}$ if  it is clear from the context what sampling is being considered. Further, let 
		\begin{equation}
		p_i\eqdef \mathbf{P}_{ii} = \mathbb{P}(i \in \hat{S}), \qquad i\in [n].
		\end{equation}

A sampling $\hat{S}$ for which $p_i>0$ for all $i\in [n]$ is called \emph{proper}. 	It is easy to see that the probability matrix $\mathbf{P}(\hat{S})$ does not uniquely determine the underlying sampling $\hat{S}$. For any matrix $\mathbf{M}\in \R^{n\times n}$, by $\mathbf{M}_{\hat{S}}$ we denote the random variable that selects out the elements of $\mathbf{M}$ according to  $\hat{S}$, as defined in \eqref{eq:slice}.

\section{Main Assumptions and Random Matrix Sampling}
We start this section with three assumptions that concern our objective function.

\subsection{Assumptions}
In the following lines we present three main assumptions on the problem structure and machinery at hand used to solve it.
\begin{assumption}[Smoothness]\label{ass:smooth}
	There exists a symmetric positive definite matrix $\mathbf{M} \in \mathbb{R}^{n \times n}$ such that $\forall x,h \in \mathbb{R}^n$,
	\begin{equation}\label{eq:smooth}
	f(x+h)\leq f(x)+\braket{\nabla f(x),h}+\frac{1}{2}\braket{h,\mathbf{M}h}.
	\end{equation}
\end{assumption}

\begin{assumption}[Strong Convexity]\label{ass:strongconvex}
	There exists a symmetric positive definite matrix $\mathbf{G} \in \mathbb{R}^{n \times n}$ such that $\forall x,h \in \mathbb{R}^n$,
	\begin{equation}\label{eq:strgcnvx}
	f(x)+\braket{\nabla f(x),h}+\frac{1}{2}\braket{h,\mathbf{G}h}\leq f(x+h)
	\end{equation}
\end{assumption}
	Minimizing \eqref{eq:strgcnvx} on both sides in $h$ gives
	\begin{equation}\label{eq:strgcnvx2}
	f(x)-f(x^*)\leq\frac{1}{2}\braket{\nabla f(x),\mathbf{G}^{-1}\nabla f(x)},
	\end{equation}
	where $x^*$ denotes the (necessarily unique) minimizer of $f$. Also note that clearly
	\begin{equation}\label{eq:ordering}
	\mathbf{G} \preceq \mathbf{M},
	\end{equation}
	with equality if and only if $f$ is a quadratic.

\subsection{Samplings}
We begin by defining samplings used in this work and by exposing the differences among them. Previous papers such as \cite{ESO, QUARTZ, NSYNC} focused on arbitrary samplings. In this work we will focus on subset of possible \emph{proper} samplings, which can be easily implemented in practice. However, for the sake of completeness, we will define other samplings as well. 

\begin{definition}\hfil
\begin{enumerate}
\item 
	$\tau$-nice sampling picks subsets of $[n]$ with cardinality $\tau$, uniformly at random. 

\item 
	$\tau$-list sampling picks subsets of $[n]$ with cardinality $\tau$, uniformly at random with constraint that subsets have to contain successive elements modulo $n$. 
\item 
	Parallel ($\tau$,c)-nice sampling performs $c$ independent $\tau$-nice samplings with replacement. The independent sets are denotes $\hat{S}_1,\hat{S}_2,\dots,\hat{S}_c$.

\item \label{def:tau-list}
	Parallel ($\tau$,c)-list sampling performs $c$ independent $\tau$-list samplings with replacement. The independent sets are denotes $\hat{S}_1,\hat{S}_2,\dots,\hat{S}_c$.
\item 
	Parallel ($\tau,c$)-non-overlapping sampling performs one $\tau c$-nice sampling on a master node. Subsequently, the set of size $c\tau$ is partitioned to c sets and distributed to worker nodes.
	\end{enumerate}
\end{definition}

\begin{remark} The difference between parallel $(c,\tau)$-nice sampling and standard $c\tau$-nice sampling can be 
	visualized by looking at what part of matrix influences a single iteration. For example, suppose that $n = 5$, $c = 2$ and $\tau = 2$. 
	
	Both samplings sample $c\tau$ coordinates; in this case $4$. Let the coordinates sampled with $c\tau$-nice sampling be $\{1,2,3,4\}$. Further, the parallel sampling samples two sets at random (let this be $\{1,2\}$ and $\{3,4\}$ for this discussion) and distributes it to the worker nodes, whereas with serial $\tau$-nice all four coordinates are handled by the single master node. 
	
	\begin{figure}
		\centering
		\includegraphics[scale=0.4]{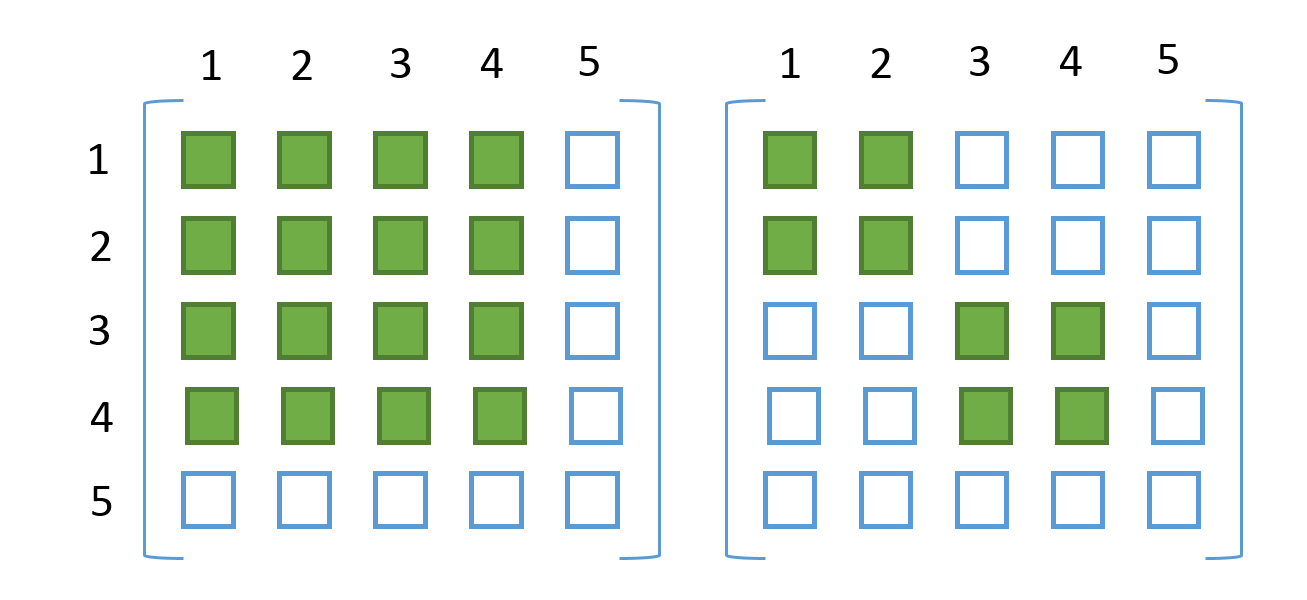}
		\caption{On the left side we see a sampling of a matrix of size $5 \times 5$ with $\tau$-nice sampling ($\tau = 4$). Similarly, on the right we see sampling with $(c,\tau)$-nice parallel sampling with $\tau=2$ and $c=2$. Namely, here the two sets sampled are $\{\{1,2\},\{3,4\}\}$. We can see that the serial sampling accesses more information in a single iteration even when the number of coordinates updated at each iteration is the same.}
		\label{fig:matrix_sampling}
	\end{figure}
	
	In Figure \ref{fig:matrix_sampling} one can see that with $\tau$-nice more information is used, however a bigger matrix has to be inverted leading to greater computational costs. 
	
\end{remark}

\begin{assumption}[Independence of Serial Samplings]
	The random sets $\{S_k\}_{k\geq0}$ are:
	\begin{enumerate}
		\item independent and identically distributed, 
		\item proper, and 
		\item non-vacuous (i.e., $\mathbb{P}(S_k=\emptyset)=0$).
	\end{enumerate}
\end{assumption}

\begin{assumption}[Independence of Parallel Samplings]\label{ass:parallel_sample}
	For the random sets of sets \\
	$\{\{S^k_1,S^k_2,\dots,S^k_c\}\}_{k\geq0}$ holds that:
	\begin{enumerate}
		\item the sets are independent and identically distributed
		\item The sets $S^k_i$ $\forall i \in [n]$ are identically and independently distributed
		\item Each of the $S^k_i$ is proper.
		\item Each of the $S^k_i$ is non-vacuous; i.e. $\mathbb{P}(S^k_i=\emptyset)=0$.
	\end{enumerate}
\end{assumption}

	\section{The Algorithm}

	\subsection{Serial algorithm for smooth functions}
	The serial method that shall be extended to parallel settings in this work was first formulated in \cite{SDNA}. The method was introduced to solve problem in $\min_{x \in \mathbb{R}^n} f(x)$, where $f$ is smooth and strongly convex. It is an iterative method where new best estimate on optimal solution in relation to the previous is given by
	\begin{equation}\label{eq:iter}
		x^{k+1}=x^{k}+h^k.
	\end{equation}
	In the original language the update step is defined via,
	\begin{equation}
		h^k=-(\mathbf{M}_{S_k})^{-1}\nabla f(x^k). 
	\end{equation}
	It has been shown that this method converges to optimum given that smoothness and strong convexity assumptions hold. In particular, the following theorem is valid. 
	
	\begin{theorem}[\cite{ESO}]\label{theorem:Serial_Method_1}
		Let Assumptions 1,2 and 4 be satisfied. Let $\{x^k\}_{k\geq0}$ be a sequence of random vectors produced by the Serial Method and let $x^*$ be optimum of function $f$. Then,
		\begin{equation}\label{eq:complexity}
			\mathbb{E}\left[f(x^{k+1})-f(x^*)\right]\leq (1-\sigma_1)\mathbb{E}\left[f(x^k)-f(x^*)\right],
		\end{equation}
		where
		\begin{equation}\label{eq:sigma_1}
			\sigma_1\eqdef\lambda_{\min}\left(\mathbf{G}^{1/2}\mathbb{E}\left[\left(\mathbf{M}_{\hat{S}}\right)^{-1}\right]\mathbf{G}^{1/2}\right).
		\end{equation}
	\end{theorem}
	
	The complexity of the algorithm depends on a parameter denoted $\sigma_1$, which is a complicated scalar parameter loosely related to the condition number of a matrix $\bM$. More intuitive analysis of this parameter is presented in \cite{SDNA}, and for special classes of problems, further in this work.

	\subsection{Parallel formulation for smooth functions}
	In this section we introduce a parallel extension of the serial method. We follow the same general iterative scheme defined by \eqref{eq:iter}, and thus our method differs only by definition of the update rule which we define as
	\begin{equation}
		h^k=-\frac{1}{b}\sum_{i=1}^{c}\left(\mathbf{M}_{S^k_i}\right)^{-1}\nabla f(x^k).\label{eq:update}
	\end{equation}

The index $k$ is the iteration counter and the index $i$ labels the subsets of $[n]$, $S^k_i$, at each iteration. One can see that our method depends on the parameter $b$, which we shall comment on  in the next sections on convergence analysis. The parameter $b$ cannot admit any values however. We show in the next section that, as long as $b$ is bigger than some threshold value $b^*$, the parallel method is, in a certain precise sense, superior to the serial method.  
	
	\begin{algorithm}
		\caption{PSN: Parallel Stochastic Newton Method}
		\label{alg:Method1}
		\renewcommand{\algorithmicrequire}{\textbf{Parameters:}}
		\renewcommand{\algorithmicensure}{\textbf{Initialization:}}
		\begin{algorithmic}[1]
			\Require sampling $\hat{S}$; data matrix $\bM$; aggregation parameter $b$
			\Ensure Pick $x^0 \in \mathbb{R}^n$
			\For{ $k = 0,1,2, \dots $ }
			
			\For{ $i = 1, \dots, c $ in parallel}
			\State Independently generate a random set  $\hat{S}^k_i \sim \hat{S}$ 
			\State $h^{k}_i \leftarrow \left(\mathbf{M}_{\hat{S}^k_i}\right)^{-1}\nabla f(x^k)$
			\EndFor
			\State $x^{k+1} \leftarrow x^k - \frac{1}{b}\sum_{i=1}^{c} h^{k}_i$ 
			\EndFor
			
		\end{algorithmic}
	\end{algorithm}

\subsection{Main convergence result}

	We are now ready to present the main theoretical result of this paper.
	
\begin{theorem}\label{theorem:Parallel_Method_1}
			Let Assumptions \ref{ass:smooth}, \ref{ass:strongconvex} and \ref{ass:parallel_sample} be satisfied. 
 Assume that the aggregation parameter $b$ satisfies
			\begin{equation}\label{eq:b}
			b \geq (c-1)\lambda \theta+1,
			\end{equation}
where
			\begin{equation} \label{eq:lambda}
			\lambda\eqdef\lambda_{\max}\left(\mathbf{G}^{-1/2}\mathbf{M}\mathbf{G}^{-1/2}\right)
			\end{equation}
			and
\begin{equation}\label{eq:a_1}
			\theta \eqdef\lambda_{\max}\left(\mathbf{G}^{1/2}\mathbb{E}\left[\left(\mathbf{M}_{\hat{S}}\right)^{-1}\right]\mathbf{G}^{1/2}\right).
\end{equation}
Then the random iterates $\{x^k\}$ produced by the PSN method (Algorithm~\ref{alg:Method1}) satisfy
			\begin{equation}
			\mathbb{E} \left[f(x^{k+1})-f(x^*) \right]\leq(1-\sigma_p)\mathbb{E} \left[f(x^k)-f(x^*)\right]
			\end{equation}
			where
			\begin{equation}\label{eq:sigma_p}
			\sigma_p\eqdef\frac{c\sigma_1}{b} = \frac{c\sigma_1}{1+ \theta \lambda (c-1)}.
			\end{equation}

		\end{theorem}

	\section{Complexity Analysis}
	In order to develop a sound complexity analysis and prove Theorem \ref{theorem:Parallel_Method_1}, we need to develop a series of lemmas dealing with expectations of matrix minors and relations of constants defined in $\eqref{eq:a_1}$ and $\eqref{eq:sigma_p}$.

	\subsection{Sampling lemmas}
	
	
	\begin{lemma}[Tower property for expectation of matrix minors]
		Let $\bA \in \mathbb{R}^{n \times n}$ be symmetric, and $x \in \mathbb{R}^n$. If $\hat{S}_1, \hat{S}_2\sim \hat{S}$ are i.i.d. samplings, then,
		\begin{equation}\label{eq:tower_prop}
			\mathbb{E}\left[\braket{\bA \bA_{\hat{S}_1}x,\bA_{\hat{S}_2}x}\right]=\braket{\bA\mathbb{E}[\bA_{\hat{S}}]x,\mathbb{E}[\bA_{\hat{S}}]x}.
		\end{equation}
	\end{lemma} 
	
	\begin{proof}
	We write
		\begin{eqnarray*}
			\mathbb{E}[\braket{\bA \bA_{\hat{S}_1}x,\bA_{\hat{S}_2}x}] 
		&=& \mathbb{E}[x^\top \bA_{\hat{S}_1}\bA \bA_{\hat{S}_2}x] \\
		&=& x^\top\mathbb{E}[\bA_{\hat{S}_1}\bA \bA_{\hat{S}_2}]x \\
			&
			= &  x^\top \mathbb{E}\left[\mathbb{E}\left[\bA_{\hat{S}_1}\bA\bA_{\hat{S}_2} \;|\; \hat{S}_2 \right]\right]x \\
			&=& x^\top \mathbb{E}\left[ \mathbb{E}[\bA_{\hat{S}_1}] \bA\bA_{\hat{S}_2}  \right]x  \\
			&= &x^\top \mathbb{E}[\bA_{\hat{S}}]\bA\mathbb{E}[\bA_{\hat{S}}]x \\
			&=& \braket{\bA\mathbb{E}[\bA_{\hat{S}}]x,\mathbb{E}[\bA_{\hat{S}}]x}.
		\end{eqnarray*}
In the second step we use linearity of expectation, combined with linearity of the mapping $\bX \to x^\top \bX x$. In the third step we use the tower property. 		In the fourth step, we use linearity of expectation and independence. In the fifth step we use linearity of expectation again, combined with the assumption that $\hat{S}_1$ and $\hat{S}_2$ have the same distribution as $\hat{S}$.
	\end{proof}
	
	\begin{lemma}\label{lemma:shur_pos}
		Assume that the block matrix
		\begin{equation}
			\bQ=\begin{pmatrix}
				\mathbf{A} & \mathbf{B} \\
				\mathbf{B}^\top & \mathbf{C} 
			\end{pmatrix}
		\end{equation} is positive definite. Then the matrix defined as 
		\begin{equation}\
			\mathbf{K}=\mathbf{A}^{-1}\mathbf{B}(\tilde{\mathbf{A}})^{-1}\mathbf{B}^\top\mathbf{A}^{-1},
		\end{equation}  where $\tilde{\mathbf{A}}$ denotes the Shur complement of  $\mathbf{A}$,
		is positive semi-definite.
	\end{lemma}
	\begin{proof}
		By \cite{Zhang2011} Notion 7.4 we know that Shur complement is positive semi-definite if $\mathbf{C}$ is non-singular. As $\mathbf{A}$ is positive definite, then $\mathbf{C}$ is non-singular. Thus, $\mathbf{K}$ must be positive semi-definite as well.
	\end{proof}

	\begin{lemma}[Zhang in \cite{Zhang2011}] \label{lemma:order}
		Let $\bM$ be a positive definite matrix, and $S$ be a subset of $[n]$, then
		\begin{equation}\label{eq:inv}
			(\mathbf{M}_S)^{-1} \preceq (\mathbf{M}^{-1})_S.
		\end{equation}
	\end{lemma}
	\begin{proof}
		We suppose that $S$ is a sampling such that it selects a sub-matrix of size $k$. The proof of the ordering \eqref{eq:inv} is equivalent to showing the result for $k \times k$ matrix principal minors $(\bM_{SS})^{-1}$ from \eqref{eq:smaller_splice}. 
		
		In the following analysis, we suppose that the sub-matrix that is sampled with $S$ is located in upper-left corner. Let us denote this sub-matrix by  $\mathbf{A}$. 
		\begin{equation}
			\mathbf{M}=\begin{pmatrix}
				\mathbf{A} & \mathbf{B} \\
				\mathbf{B}^\top & \mathbf{C} 
			\end{pmatrix}
		\end{equation}
		Also let matrix $\mathbf{M}^{-1}$ have a s similar block decomposition
		\begin{equation}
			\mathbf{M}^{-1}=\begin{pmatrix}
				\mathbf{X} & \mathbf{Y} \\
				\mathbf{Y}^\top & \mathbf{Z} 
			\end{pmatrix}.
		\end{equation}
		Then we know that $\mathbf{X}$ is related to Shur complement of $\mathbf{A}$  by \cite[Theorem 2.4]{Zhang2011}. We denote the Shur complement of $\mathbf{A}$ as  $\tilde{\mathbf{A}}$. Hence, $\mathbf{X}=\mathbf{A}^{-1}+\mathbf{A}^{-1}\mathbf{B}\tilde{\mathbf{A}}^{-1}\mathbf{B}^\top \mathbf{A}^{-1}=\mathbf{A}^{-1}+\mathbf{K}$. Then as $\mathbf{K}$ is positive definite by Lemma \ref{lemma:shur_pos} we have
		\begin{equation} \label{eq:fefefefe}
			((\mathbf{M}^{-1})_{SS})=\mathbf{X}=\mathbf{A}^{-1}+\mathbf{K} \succeq \mathbf{A}^{-1} = ((\mathbf{M}_{SS})^{-1}).
		\end{equation}
		Should $S$ be a sampling that does not sample a principal sub-matrix in the upper left corner then we can consider the matrix $\mathbf{Z}=\mathbf{Q}^\top \mathbf{M}\mathbf{Q}$ where $\mathbf{Q}$ is a matrix that permutes the elements such that the desired sub-matrix is in the upper left position. Then we can define a sampling $S'$ s.t. $\mathbf{I}_{:S}= \mathbf{Q}\mathbf{I}_{:S'}$, which yields	
		\begin{equation*} 
			(\mathbf{M}_{SS})^{-1}\stackrel{\eqref{eq:smaller_splice}}=(\mathbf{I}_{:S}^\top \mathbf{M}\mathbf{I}_{:S})^{-1}=(\mathbf{I}_{:S'}^\top \mathbf{Q}^\top \mathbf{M}\mathbf{Q}\mathbf{I}_{:S'})^{-1} =(\mathbf{Z}_{S'S'})^{-1} \stackrel{\eqref{eq:fefefefe}}\succeq (\mathbf{Z}^{-1})_{S'S'}=(\mathbf{M}^{-1})_{SS}.
			\end{equation*}
	\end{proof}
	
		\begin{lemma}
			Let $S'$ and $S$ be two random valued samplings s.t. $S' \subset S \subset [n] $. Also let $\mathbf{X} \in \mathbb{R}^{n \times n}$ be a positive definite matrix.
			\begin{equation}\label{eq:nested}
			(\mathbf{X}_{S'})^{-1} \preceq (\mathbf{X}_{S})^{-1}
			\end{equation}
		\end{lemma}
		\begin{proof}
 			Let $ \mathbf{Y} = \mathbf{X}_{SS}$ then, we can rewrite $\mathbf{X}_{S'} = \bI_{:S'}(\mathbf{I}_{S'S}\mathbf{Y}\mathbf{I}_{SS'})\bI_{S':} = \mathbf{Y}_{S'}$. By Lemma \ref{lemma:order} and positive definiteness, we know that $ (\mathbf{X}_{S'})^{-1}= (\mathbf{Y}_{S'})^{-1} \stackrel{\eqref{eq:inv}} \preceq (\mathbf{Y}^{-1})_{S'} \stackrel{\text{pos. def.}}\preceq (\mathbf{Y}^{-1})_{S} = (\mathbf{X}_{S})^{-1}. $
		\end{proof}
%
	\subsection{Parallelization parameters}

	\begin{lemma}
		\begin{equation}\label{eq:firstorder}
			0 \leq \sigma_1 \leq \theta \leq 1
		\end{equation}
	\end{lemma}
	\begin{proof}
		The proof of the first inequality follows by noting that $\sigma_1$ is the smallest eigenvalue of a positive definite matrix; the second is by the definition of $ \theta$ and $\sigma_1$. The last inequality follows from the convexity of $\lambda_{\max}$ operator and Jensen's inequality, namely,
		\begin{equation}\label{eq:jensen}
			\mathbb{E}[\lambda_{\max}(\mathbf{X})]\geq \lambda_{\max}(\mathbb{E}[\mathbf{X}]).
		\end{equation}
		\begin{eqnarray}
			\theta & \stackrel{\eqref{eq:a_1}}  = & \lambda_{\max}(\mathbf{G}^{1/2}\mathbb{E}[(\mathbf{M}_{\hat{S}})^{-1}]\mathbf{G}^{1/2}) \stackrel{\eqref{eq:inv}}
			\leq  \mathbb{E}[\lambda_{\max}(\mathbf{G}^{1/2} (\mathbf{M}^{-1})_{\hat{S}} \mathbf{G}^{1/2})] \\
			& \stackrel{\eqref{eq:nested}}\leq & \mathbb{E}[\lambda_{\max}(\mathbf{G}^{1/2} (\mathbf{M}^{-1})\mathbf{G}^{1/2})] =  \mathbb{E}\left[\frac{1}{\lambda_{\min}(\mathbf{G}^{-1/2} (\mathbf{M}) \mathbf{G}^{-1/2})} \right] \stackrel{\eqref{eq:ordering}} \leq 1 \notag
		\end{eqnarray}
	\end{proof}
	
	\begin{remark}
		When $\bM = \bG$ (quadratic cost function), as $ \theta<1$, then $b \leq c$. Consequently, $\sigma_p \geq \sigma_1$.
	\end{remark}
	\hfil \\
	\begin{proposition}[Bound on $ \theta$ for $\tau$-list samplings]\label{prop:bound}
		Let $\bM=\bG$ as in \eqref{eq:smooth} and $\hat{S}$ be parallel $(\tau,c)$-list sampling, then 
		\begin{equation}
		 \theta \leq \frac{\tau}{n} \operatorname{cond}(\bM).
		\end{equation}
	\end{proposition}
	
	\begin{proof}
		
		\begin{eqnarray}
			 \theta&\stackrel{\eqref{eq:a_1}} =& \lambda_{max}(\bM^{1/2}\mathbb{E}[(\bM_{\hat{S}})^{-1} ]\bM^{1/2})\\
			   &\stackrel{} =& \max_{\norm{x}=1} \braket{\bM^{1/2}x,\mathbb{E}[(\bM_{\hat{S}})^{-1}]\bM^{1/2}x}\\
			   &\stackrel{} =& \max_{\norm{\bM^{-1/2}y}=1} \braket{y,\mathbb{E}[(\bM_{\hat{S}})^{-1}]y}\\
			   &\stackrel{\text{Def. \ref{def:tau-list}}} =& \max_{\norm{M^{-1/2}y}=1} \frac{1}{n}\sum_{i=1}^{n}\braket{y,(\bM_{S_i})^{-1}y}		  \\
			   & = & \max_{\norm{M^{-1/2}y}=1} \frac{1}{n}\sum_{i=1}^{n}\braket{\bI_{S:}y,(\bM_{S_iS_i})^{-1}\bI_{S:}y}		  \\
			   & \stackrel{\eqref{eq:nested}} \leq & \max_{\norm{M^{-1/2}y}=1} \frac{1}{n}\sum_{i=1}^{n} \lambda_{max}((\bM)^{-1})\norm{\bI_{S:}y}^2\\
			   & = & \frac{\tau}{n} \lambda_{max}(\bM^{-1}) \max_{\norm{M^{-1/2}y}=1} \norm{y}^2\\
			   & = & \frac{\tau}{n}\operatorname{cond}(\bM)
		\end{eqnarray}
	\end{proof}
	Proposition \ref{prop:bound} can be useful for classes of problems where we have information about the condition number of the matrix $\bM$ in Assumption \ref{ass:smooth}, and $n \gg \operatorname{cond}(\bM)$. In these circumstance, the bound can be used as a good proxy for $\theta$. A prime example of such problem class are banded Toeplitz matrices where condition number is usually independent of $n$ but rather depends on the band width.

	\section{Analysis and Comparison with Existing Methods}
	\subsection{Theoretical comparison PSN with PCDM}
	
	Parallel coordinate descend method (PCDM) is a powerful parallel optimization algorithm analyzed in \cite{PCDM} and \cite{Richtarik2012}. We would like to compare the performance of this method to the current parallel method. The complexity of PCDM can be expressed using the language of this paper and the paper \cite{SDNA} by,
	\begin{equation}\label{eq:sigma_3}
		\sigma_3 = \lambda_{\min}(\mathbf{G}^{1/2}\mathbf{D}(v^{-1})\mathbf{D}(p)\mathbf{G}^{1/2}).
	\end{equation}
	
	To have a fair comparison, while using parallel $(\tau,c)$-nice sampling for PSN, this in turn would corresponds to $(\tau c)$-nice sampling for PCDM such that access the same number of coordinates is maintained. Also, the assumptions of the PCDM algorithm are a little different in what we assumed here. For example, in \cite{PCDM}, they assume that matrix $\mathbf{M}$ in the quadratic over-approximation has a decomposition such that $\mathbf{M} = \mathbf{A}^\top \mathbf{A}$. The the condition for $v$ in expression \eqref{eq:sigma_3} can be deduced from basic considerations where $v = \lambda_{\max}(\mathbf{A}^\top \mathbf{A})$, which defines $\sigma_b$, or by using structural sparsity with \cite[Proposition 5.1]{ESO}, we can calculate vector $v$ as follows
	\begin{equation} \label{eg:pcdmeq}
		v_i = \sum_{j=1}^{m}\left(1+ \frac{(|J_j|-1)(\tau c -1)  }{\max(n-1,1)}\right)\bA_{ji}
	\end{equation}
	where $J_j$ denotes $J_j:=\{i\in[n] \;|\; \bA_{ji}\neq 0\}$. 
	
	As we do not have the decomposition at disposal currently. Let us assume that $\bA$ is fully dense, and thus $|J_j| = n $ $\forall j$. In this settings, $v_i$ reduces to $v_i = \tau c \mathbf{M}_{ii}$, and $p_i = \frac{\tau c}{n}$. Thus we can see that in fact $\sigma_3$ does not depend on $\tau c $ at all, and we do not gain any theoretical speedup in convergence rate. This exemplifies the theoretical contribution made in \cite{SDNA} and this paper, which show the improvement in convergence rates without sparsity patterns as $\sigma_1$ and $\sigma_p$ depend on $\tau$ or $c$ even in the fully dense scenario. To illustrate this, we calculate the value of $\sigma_p$ for a random matrix ($n=400$) in the table \ref{tbl:tbl1}.
	
	\begin{figure}
		\centering
		\begin{tabular}{ | l | l | l | p{1cm} |  p{1cm} |p{1cm} |}
			\hline
			$c$ (cores) & $\tau$ & $\tau c$ & $\sigma_b \times 10^{-6}$ & $\sigma_3 \times 10^{-6}$ & $\sigma_p \times 10^{-6}$ \\ \hline 
			1 & 3 & 3 & 0.0015 & 0.0019 & 0.0058 \\  \hline
			2 & 3 & 6 & 0.0029 & 0.0019 & 0.0113\\ \hline
			4 & 3 & 12 & 0.0058 & 0.0019 & 0.0214\\  \hline
			8 & 3 & 24 & 0.0117 & 0.0019 & 0.0386\\  \hline
			16 & 3 & 48 & 0.0233 & 0.0019 & 0.0646\\  \hline
			32 & 3 & 96 & 0.0466 & 0.0019 & 0.0975\\  \hline
			64 & 3 & 192 & 0.0933 & 0.0019 & 0.1308\\  \hline
			128 & 3 & 384 & 0.1866 & 0.0019 & 0.1578\\  \hline
		\end{tabular}
		\caption{The theoretical convergence rates for quadratic problems with $\mathbf{A}^\top \mathbf{A} = \mathbf{M}$, where $\mathbf{A}$ was $400 \times 400$ matrix. The entries of the matrix $\bA$ were sampled from a standard normal distribution. The basic approach with the largest eigenvalue $\sigma_b$ works better only when $\tau c$ is very big, however given the difficulty of computing inverting large matrices, this approach is often impractical.}
		\label{tbl:tbl1}
	\end{figure}
	
	\begin{figure}
		\centering
		\begin{tabular}{ | l | l | l | p{1cm} |  p{1cm} |p{1cm} |}
			\hline
			$c$ (cores) & $\tau$ & $\tau c$ & $\sigma_b \times 10^{-6}$ & $\sigma_3 \times 10^{-6}$ & $\sigma_p \times 10^{-6}$ \\ \hline
			1 & 3 & 3 & 0.0604 & 0.1744 & 0.2679 \\ \hline
			2 & 3 & 6 & 0.1208 & 0.2305 & 0.5199\\ \hline
			4 & 3 & 12 & 0.2416 & 0.2747 & 0.9818\\ \hline
			8 & 3 & 24 & 0.4833 & 0.3038 & 1.7664\\ \hline
			16 & 3 & 48 & 0.9666 & 0.3208 & 2.9418\\ \hline
			32 & 3 & 96 & 1.9332 & 0.3300 & 4.4086\\ \hline
			64 & 3 & 192 & 3.8663 & 0.3348 & 5.8727\\ \hline
			128 & 3 & 384 & 7.7326 & 0.3373 & 7.0420\\ \hline
		\end{tabular}
		\caption{The theoretical convergence rates for quadratic problems with  $\mathbf{A}^\top \mathbf{A} = \mathbf{M}$ where $\mathbf{A}$ was $400\times400$ matrix. The matrix $\mathbf{A}$ has also sparse structure where only approximately one third of the elements were non-zero. The entries were sample from a standard normal distribution. The most basic approach with the largest}
		\label{tbl:tbl2}
	\end{figure}
	
	However, one has to bear in mind that we assumed that the matrix $\bA$ was fully dense, which is the worst case scenario for the PCDM algorithm. Unfortunately, as there is not an easy way to present a comprehensive theoretical comparison, we only present a one based on generated random matrices which have either dense or sparse structure. In an experiment with sparse matrix (only one third of the elements are present) and using $\eqref{eg:pcdmeq}$, we arrived at the results summarized in Table \ref{tbl:tbl2}. These values can be almost directly compared as the cost of inversion of $3 \times 3$ matrix is too small to influence the cost of one iteration significantly. The instruction level parallelism can be utilized to greater extent as more computation is done with the same amount of information leading to better operational intensity \cite{Williams2009}.
	
	\subsection{$\rho$-matrix analysis}
	In the following analysis, we compute the convergence rates for the serial and parallel algorithm, applied to a specific problem, exactly. We minimize the quadratic function $\frac{1}{2}x^\top \mathbf{M} x $ where $\mathbf{M}$ has a special structure that we call $\rho$-matrix. The reason for introducing this special type of problem is that for this problem an analytical expression for $\sigma_1$ and $ \theta$ could be found and intuition about behavior of these theoretical constants can be illustrated fully. Since the parallel speedup depends only on parameter $ \theta$ we can predict the theoretical speedup for this class of matrices easily. The theoretical speedup is presented in Figure \ref{fig:theoretical_comp}.
	 
	\begin{definition}
		We define a $\rho$-matrix to be any $n \times n $ matrix that has the following structure,
		$$\mathbf{M} =
		\begin{pmatrix}
		1 & \rho & \cdots & \rho \\
		\rho & 1 & \cdots & \rho \\
		\vdots  & \vdots  & \ddots & \vdots  \\
		\rho & \rho & \cdots & 1
		\end{pmatrix}$$
		with $0<\rho<1$.
	\end{definition}
	
	\begin{proposition}
		Suppose we have a $\rho$-matrix $\mathbf{M}$ with $0<\rho<1$ as a parameter. Then $\sigma_1$ and $ \theta$ can be expressed as functions of $n$ and $\tau$ for $\tau$-nice sampling as follows 
		\begin{equation}
			\sigma_1=(1-\rho)\left(1-\frac{(\tau-1)B(\tau)}{(n-1)A(\tau)}\right)\frac{\tau A(\tau)}{n}
		\end{equation}
		and 
		\begin{equation}
			 \theta=(n\rho-\rho+1)\left(n\frac{(\tau-1)B(\tau)}{(n-1)A(\tau)}-\frac{(\tau-1)B(\tau)}{(n-1)A(\tau)}+1\right)\frac{\tau A(\tau)}{n}
		\end{equation}
		where we use function
		\begin{equation}
			A(\tau)=\frac{((\tau-2)\rho+1)}{(1-\rho)((\tau-1)\rho+1)}
		\end{equation}
		and 
		\begin{equation}
			B(\tau)=-\frac{\rho}{(1-\rho)((\tau-1)\rho+1)}.
		\end{equation}
	\end{proposition}
	\begin{proof}
		There are $n$ eigenvalues of the matrix $\mathbf{M}$; $n-1$ of them are degenerate and equal to $1-\rho$. This can be verified by plugging the suitable eigenvector which has only two non-zero elements ${-1,1}$. The last eigenvalue can be computed using the trace of the matrix. $\lambda=n-(1-\rho)(n-1)$. After solving for this eigenvalue, we immediately see it is the biggest one, thus, $\lambda_{\max}(\mathbf{M})=n\rho-\rho+1$.
		
		Given a $\tau$-nice sampling with $\tau$ as parameter, the matrix that is sub-sampled is a $\rho$-matrix of the size $\tau \times \tau$. Due to the symmetry of the problem, any subset of $\tau$ coordinates slices out the same matrix. Additionally, the cofactor matrices for diagonal elements are all equal, and the cofactor matrices for off-diagonal elements are all equal as well.
		
		As the entries in the inverse of the matrix depend on the determinant of the cofactor matrix only, all diagonal and non-diagonal elements of the inverse are equal. In addition, taking expectation does preserve this structure as the $\tau$-nice sampling is uniform. Thus, Let us denote the value of these two entries in diagonal and off-diagonal in the sampled inverted matrix as $A(\tau)$ and $B(\tau)$ respectively.
		
		Therefore, the value of $\mathbb{E}[\mathbf{M}_{\hat{S}}^{-1}]=\frac{\tau}{n}\bI	A(\tau)+\frac{\tau(\tau-1)}{n(n-1)}(\mathbf{S}-\bI)B(\tau)$, where $\mathbf{S}$ is matrix full of ones. To determine the values of $A$ and $B$, we compute inverse for a given $\tau$ using Cramer's rule (using matrix minors)
		
		\begin{equation}\label{eq:cramer}
			\bM^{-1}=\frac{1}{\det(\bM)}\mathbf{C}^\top,
		\end{equation}
		where $\mathbf{C}$ is matrix of cofactors. 
		
		The trick here is that we need to look only at two cofactors due to the structure of the matrix. So for $A(\tau)$ on diagonal we take $\mathbf{C}_{11}$ minor. We adopt notation where we use subscript of $\mathbf{M}$ to denote the size of the sub-matrix of $\rho$-matrix $\mathbf{M}$. We calculate determinant as a product of the known eigenvalues discussed earlier. 
		\begin{equation*}
			A(\tau)=(\mathbf{M}_{\tau}^{-1})_{11}\stackrel{\eqref{eq:cramer}}=\frac{\det(\mathbf{M}_{\tau-1})}{\det(\mathbf{M}_{\tau})}=\frac{(1-\rho)^{(\tau-2)}((\tau-2)\rho+1)}{(1-\rho)^{(\tau-1)}((\tau-1)\rho+1)}=\frac{((\tau-2)\rho+1)}{(1-\rho)((\tau-1)\rho+1)} 
		\end{equation*}
		
		When we look at $B(\tau)$ the minor of the matrix that determines any off-diagonal is a matrix which contains only $\rho$'s at first row and continues with other rows as in matrix $\mathbf{M}_\tau$. Let us denote such type of matrix $\mathbf{N}_{\tau}$, where the subscript $\tau$ symbolizes the size of the matrix again. For $\det(\mathbf{N}_{\tau})$, we arrive at the following recurrence relation by using Laplace decomposition of determinants.
		\begin{eqnarray} \label{eq:reccur}
			\det(\mathbf{N}_{\tau})& =& (1-\rho)\mathbf{N}_{\tau-1} \\
			\det(\mathbf{N}_{2})& = & \rho^2-\rho
		\end{eqnarray}
		Therefore,
		\begin{eqnarray}
			B(\tau)&=&(\mathbf{M}_{\tau}^{-1})_{12}=-\frac{\det(\mathbf{N}_{\tau})}{\det(\mathbf{M}_{\tau})}\stackrel{\eqref{eq:reccur}}=-\frac{(1-\rho)^{(\tau-2)}\rho}{(1-\rho)^{(\tau-1)}((\tau-1)\rho+1)} \notag \\&= &-\frac{\rho}{(1-\rho)((\tau-1)\rho+1)} 
		\end{eqnarray}
		
		The matrix $\mathbb{E}[\mathbf{M}_{\hat{S}}^{-1}]$ can be manipulated to the from of $\rho$-matrix by dividing by the factor $\frac{\tau A(\tau)}{n}$. Thus, we can define a new $\rho_N$ which serves as parameter of the new $\rho$-matrix. 
		
		This parameter allows us to compute the eigenvalues of the matrix, and thus, also maximal and minimal eigenvalues. From the $\sigma_1$ definition, $\sigma_1$  is the smallest eigenvalues of a product of two $\rho$-matrices one with $\rho$ as parameter and one with $\rho_N$ as the parameter. The product of two $\rho$-matrices is a $\rho$-matrix again, and as consequence of this, the smallest eigenvalue of the product is just a product of the two smallest eigenvalues ones.
		\begin{equation}
			\sigma_1=(1-\rho)(1-\rho_N)\frac{\tau A(\tau)}{n}=(1-\rho)\left(1-\frac{(\tau-1)B(\tau)}{(n-1)A(\tau)}\right)\frac{\tau A(\tau)}{n}.
		\end{equation}
		And similarly $ \theta$, and its definition imply that $ \theta$ is product of the two biggest eigenvalues
		\begin{eqnarray}
			 \theta&=&(n\rho-\rho+1)(n\rho_N-\rho_N+1)\frac{\tau A(\tau)}{n} \notag\\&=&(n\rho-\rho+1)\left(n\frac{(\tau-1)B(\tau)}{(n-1)A(\tau)}-\frac{(\tau-1)B(\tau)}{(n-1)A(\tau)}+1\right)\frac{\tau A(\tau)}{n}	.
		\end{eqnarray}	
	\end{proof}
	
	\begin{figure}[htbp]
		\centering	
		\includegraphics[width=0.80\textwidth]{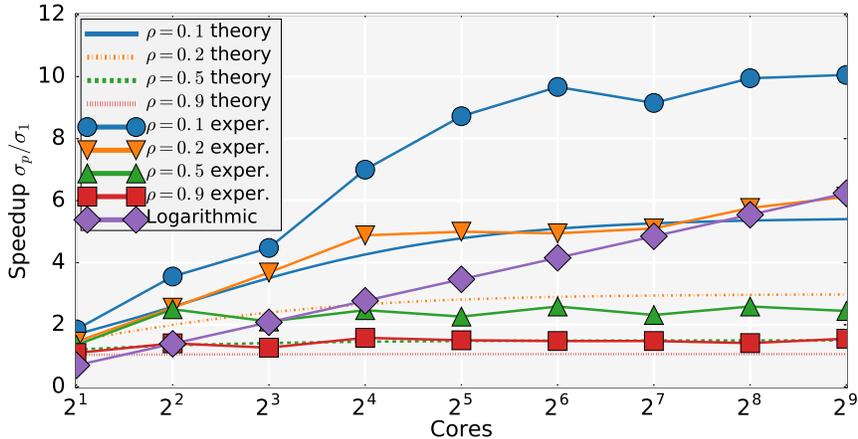}
		
		\caption{Theoretical speedup for problem of minimizing quadratic function $f(x)=\frac{1}{2}x^\top \mathbf{M} x$ with $\mathbf{M}$ being $\rho$-matrix. The size of the problem is $n=1024$, and parallel $(2,c)$-nice sampling. We observe the smaller the $\rho$, the greater the speedup. Also, we see that our estimate of speedup is rather conservative for small values of $\rho$. Intuitively, the problem should be more difficult as $\rho$ increases since $\rho$ is related to the condition number of $\mathbf{M}$, namely $\operatorname{cond}(\mathbf{M}) = \frac{n\rho - \rho-1}{1-\rho}$, which is an increasing function of $\rho$.}
		\label{fig:theoretical_comp}
	\end{figure}

	\subsection{$\alpha$-tridiagonal matrix analysis}
	
	There is only a limited number of possible special matrix structures that are simple enough to be parametrized by one parameter, and at the same time function as a useful practical model. We choose to model tridiagonal Toeplitz matrices as the next model class. As the optimization method depends only on the ratio between eigenvalues we scale the matrix such that the diagonal entries are equal to $1$. Matrices with a special structure such as banded, pentadiagonal or even tridiagonal matrices occur frequently in finite difference or finite element schemes, and are not only toy examples. 
	
	\begin{definition}[$\alpha$-tridiagonal matrix]
		Let $\alpha \in [0,\frac{1}{2})$, then $\mathbf{T}(\alpha) \in \mathbb{R}^{n \times n}$ 
		\begin{equation}
		\mathbf{T}(\alpha) = \begin{pmatrix}
		1 & \alpha & 0 & 0 & \dots & 0 \\
		\alpha & 1 & \alpha & 0 & \dots & 0 \\
		0 & \alpha & 1 & \alpha & \dots & 0 \\
		\vdots  & \vdots  & \ddots & \ddots & \ddots & \vdots \\
		0 &  \dots & 0 & 0 & \alpha & 1 \\
		\end{pmatrix}
		\end{equation}
				is a $\alpha$-tridiagonal positive-definite matrix.
	\end{definition}
	Similarly as in the previous section we minimize a quadratic function $\frac{1}{2}x^\top \mathbf{T}(\alpha)x$ and look at the theoretical speedup that we are able to achieve in Figure \ref{fig:tiger}. Due to the sparse nature of the matrix, we choose $\tau$-list matrix sampling with $\tau = 2$. The simple structure allows for explicit calculation of $\mathbb{E}[((\mathbf{T}(\alpha))_{\hat{S}})^{-1}]$, which results in pentadiagonal matrix of particular form for which eigenvalues can be efficiently computed algorithmically even for large  matrix sizes, or estimated using the following proposition. The dependence, which confirm the Proposition \ref{prop:tridiag}, is visualized in Figure \ref{fig:gull}.

	\hfil \\
	
	\begin{proposition} \label{prop:tridiag}
		Let $\mathbf{T}(\alpha)$ be a $n\times n$ $\alpha$-tridiagonal $(\alpha \in [0,\frac{1}{2}])$ matrix. Then for $2$-list parallel or serial sampling the parameter $ \theta$ is bounded from above by,
		\begin{equation}\label{eq:geo}
			 \theta \leq \frac{2}{(1-\alpha)n}.
		\end{equation}
	\end{proposition}
	\begin{proof}
		First we observe that the matrix $\mathbf{T}_t = \mathbb{E}[((\mathbf{T}(\alpha))_{\hat{S}})^{-1}]$ has a special tridiagonal structure with different elements only in two entries. By multiplying together $\mathbf{T}(\alpha)\mathbf{T}_t$ we obtain matrix whose eigenvalue spectrum is the same as the one of $\mathbf{T}(\alpha)^{1/2}\mathbf{T}_t\mathbf{T}(\alpha)	^{1/2}$, due to cyclic property of $\lambda_{max}$. The rest of the proof applies Gershgorin circle theorem \cite{Gerschgorin1931}, which bounds eigenvalues. The resulting matrix  $\mathbf{T}(\alpha)\mathbf{T}_t$ is constant on the diagonal and the inequalities arising from Gershgorin circle theorem are nested, thus, we are left with only one condition in \eqref{eq:geo}.
	\end{proof}
	
	The Proposition \ref{prop:tridiag} reveals that already in the worst case $\alpha = \frac{1}{2}$, we can show parallel speedup for matrices of size $n \geq 5$ even with the crude bound presented. Should, we want to solve large tridiagonal system as $\min_x \frac{1}{2}\norm{\mathbf{T}(\alpha)x - b}^2$, we can use Proposition \ref{prop:bound} with explicit formula for eigenvalues of tridiagonal matrix from \cite{Noschese2013} to bound the eigenvalues for any $n$, $\lambda_{max}(\mathbf{T}(\alpha)) \leq 1 + 2\alpha$ and $\lambda_{min}((\mathbf{T}(\alpha)))\geq 1 - 2\alpha$. Thus, the bound for $ \theta$ becomes $ \theta \leq \frac{\tau}{n}\frac{(1+2\alpha)^2}{(1-2\alpha)^2}$, which for sufficiently big $n$ leads to theoretical parallel speedup again.
	
	The utilization of PSN for tridiagonal matrices for small problems is unlikely as there exists a very efficient $O(n)$ algorithm known in literature as Thomas algorithm \cite{Thomas1949}. However similar approaches, as presented here, can be applied to general banded matrices if they have special structure, and the bound on eigenvalues can be provided by analytical means or via Gershgorin circle theorem.
		
		\begin{figure}
			\centering
			
			\begin{subfigure}[t]{0.80\textwidth}
				\includegraphics[width=\textwidth]{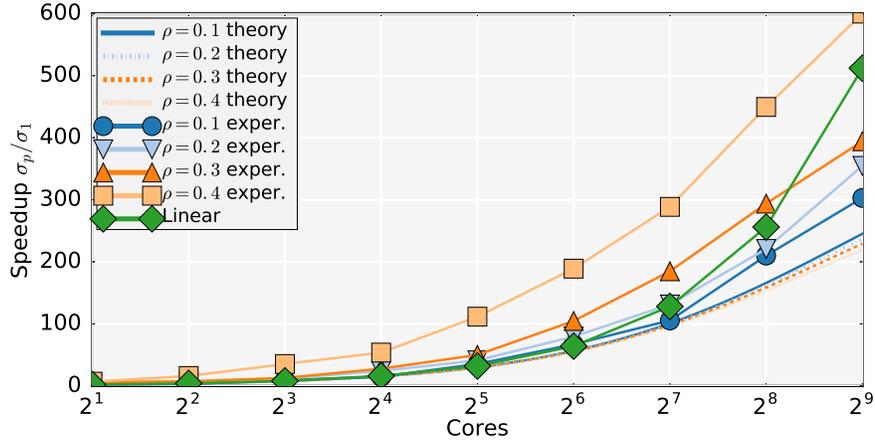}
				\caption{Theoretical and practical speedup for the class of $\alpha$-tridiagonal matrices. }
				\label{fig:gull}
			\end{subfigure}
			\hfil
			\begin{subfigure}[t]{0.80\textwidth}
				\includegraphics[width=\textwidth]{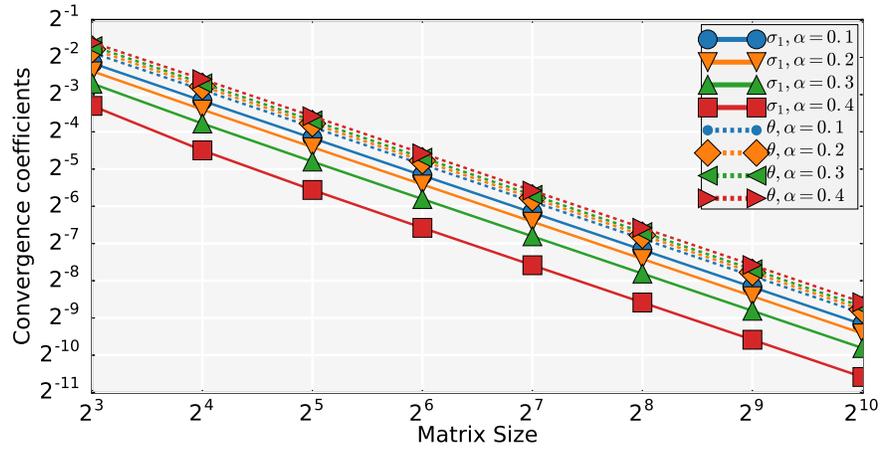}
				\caption{Values of $\sigma_1$ and $ \theta$ for $\alpha$-tridiagonal matrices.}
				\label{fig:tiger}
			\end{subfigure}

			\caption{Properties of problems with $\alpha$-tridiagonal matrices.}
			\label{fig:1}
		\end{figure}

	\section{Numerical performance}
	We performed two experiments to demonstrate the parallelization speedup of the algorithm in practice. We first compare the serial and the parallel method on minimization of $\frac{1}{2}\norm{\bX x - y}^2$ with respect to $x$, where $\bX$  is fully dense, and show superior convergence properties in terms of iteration. We report empirical speedup of the methods in Figure \ref{fig:fig4}. The $\theta$ value for this experiment was handpicked to be $0.7$.
	
	 Secondly, we compare our parallel method with PCDM. While the PSN was run with specific parallel $(\tau,c)$-nice sampling all the PCDM codes were run with $\tau c$-nice sampling to ensure fair comparison. The PSN algorithm was run with a handpicked value of $ \theta$ equal to $0.7$ in the experiments below. The main comparison of quadratic function minimization as in the previous experiment is presented in Figure \ref{fig:gull2}. We report moderate improvement on \emph{mushroom} dataset in Figure \ref{fig:tiger2}. 
	 
	All artificial data for the dense matrix $\bX$ and $y$ were generated by sampling standard normal distribution. The experiments were run on Lenovo T450S Intel Core i7 (Broadwell) with 2.6 GHz cores, and the C++ code was compiled with GCC compiler version 5 and  directives \emph{-O3 -mfma}. 	
	\begin{figure}
		\centering
		
		\begin{subfigure}{0.80\textwidth}
			\includegraphics[width=\textwidth]{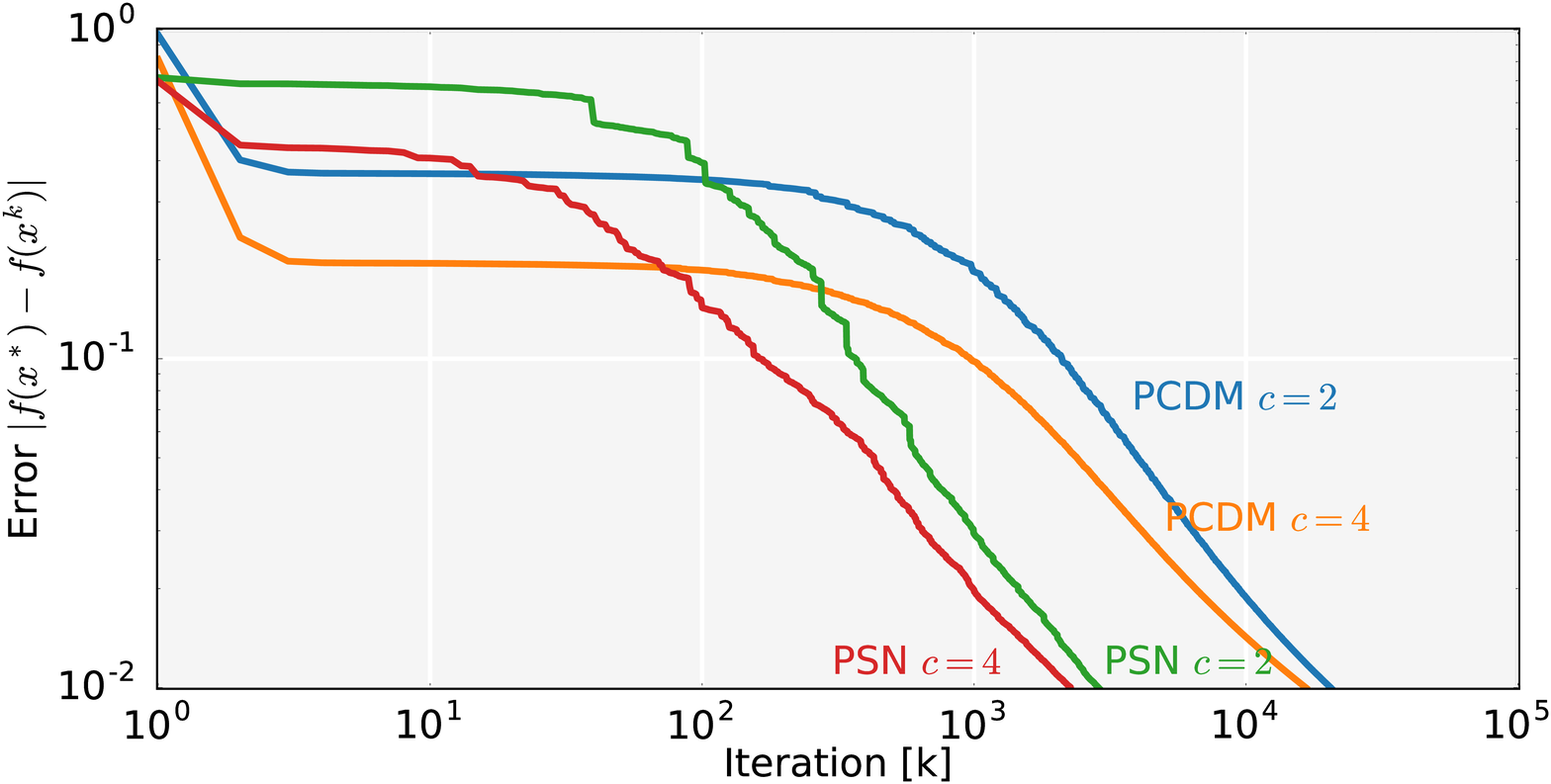}
			\caption{A linear regression model with an artificial dataset such that $n = m = 10^3 $, and $(3,c)$-nice sampling. The parameter $c$ in the graph represents the number of computational units. }
			\label{fig:gull2}
		\end{subfigure}
		\hfil
		\begin{subfigure}{0.80\textwidth}
			\includegraphics[width=\textwidth]{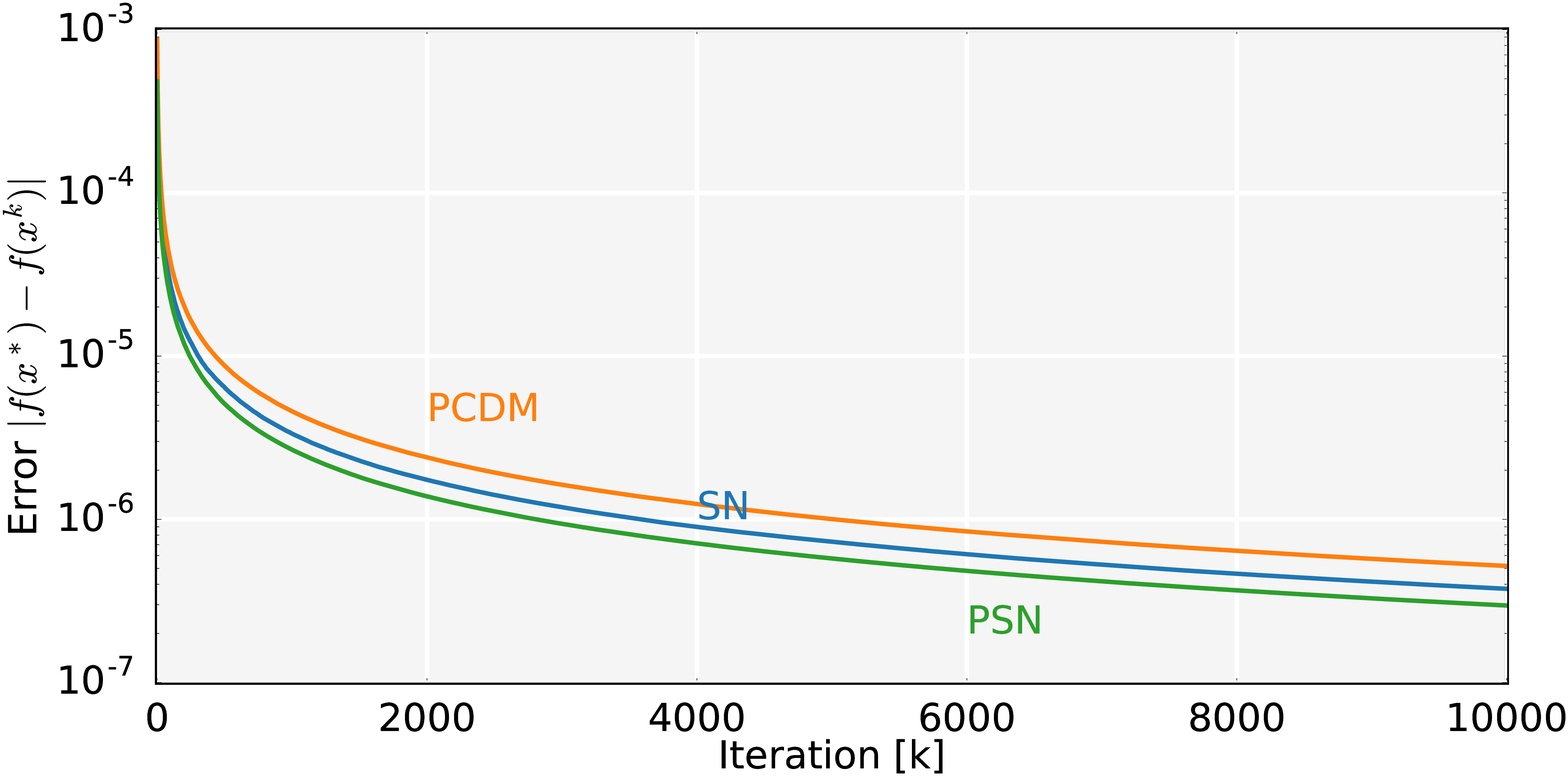}
			\caption{A logistic regression model based on \emph{mushroom} dataset with $n = 8124$, $m = 22$ and $(2,c)$-nice sampling. We run PSN with $1$ (SN) and $4$ cores, and PCDM with $\tau c$-nice sampling with the constant $\tau c = 8$. We scaled the eigenvalues of the problem before optimization.}
			\label{fig:tiger2}
		\end{subfigure}
		\caption{Performance of PSN versus PCDM on linear and logistic regression problems. }
		\label{fig:3}
	\end{figure} 
	
	\begin{figure}
		\centering
		\includegraphics[width=0.80\textwidth]{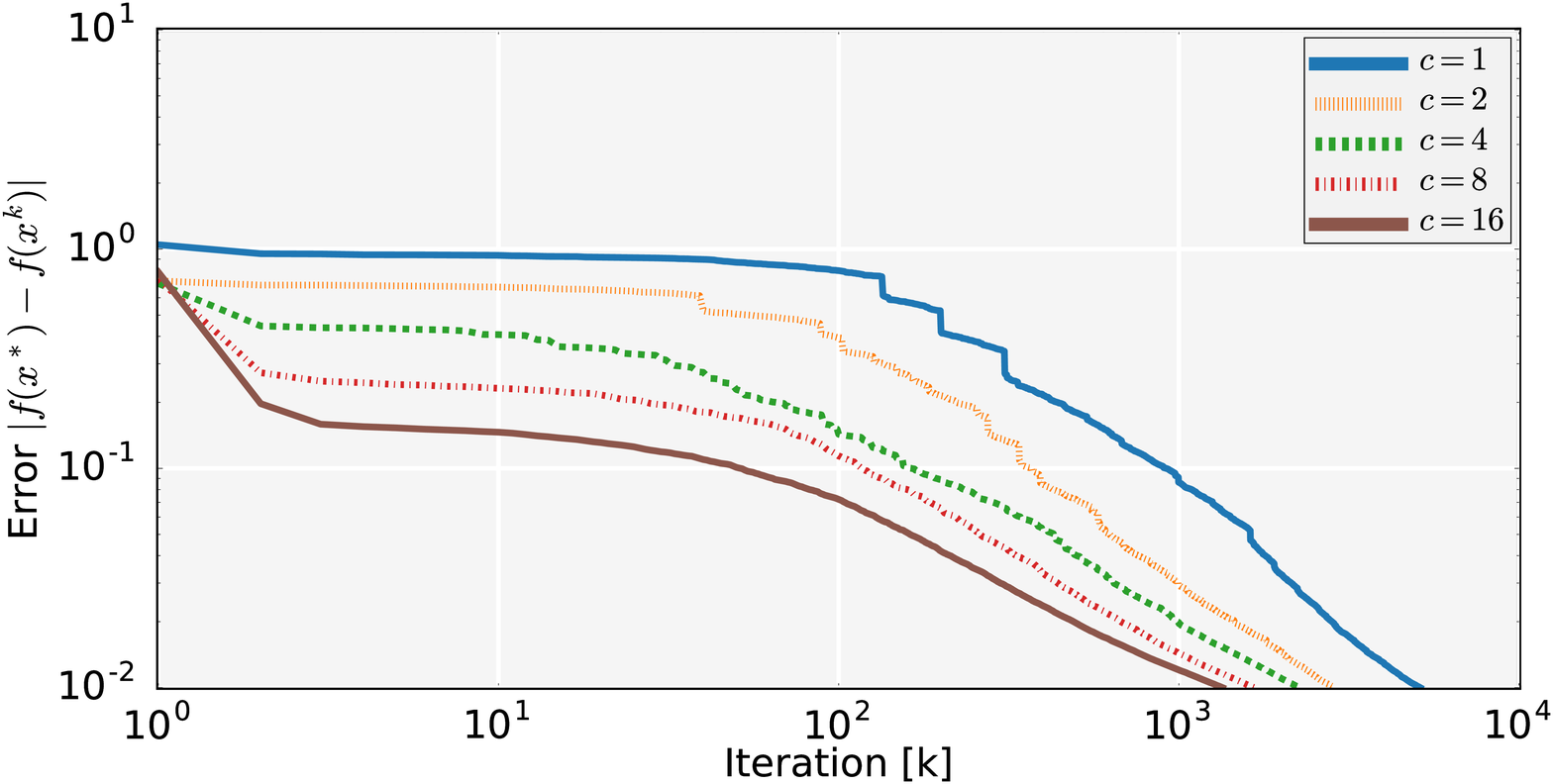}
		\caption{Comparison of a serial and a parallel stochastic Newton method. We run the parallel method with a different number of computational units. In all cases, we fit a linear regression with $n=m=10^3$ examples and variables. We utilize the $(3,c)$-nice sampling.}
		\label{fig:fig4}
	\end{figure}

	In addition, we perform a numerical experiment that arises in numerical analysis when solving the heat equation with finite differences using a high-order scheme. Namely, we simulate $\partial_t u(x,t) = \partial_{xx}u(x,t)$ on $x \in [-L,L]$ from $t \in [0,T]$. The initial condition is $u(x,0) = \cos(\frac{\pi x}{2L})$. We use the following finite difference scheme to discretize the equations
	
	\begin{eqnarray*}
		\partial_t u(x,t) & =  & \frac{u(x,t+\Delta t/2)- u(x,t-\Delta t/2)}{\Delta t} \\
		\partial_{xx} u(x,t) & = & \frac{  -\frac{1}{12} u(x+2h,t) + \frac{4}{3} u(x+h,t)  -\frac{5}{2} u(x,t) +\frac{4}{3} u(x-2h,t) -\frac{1}{12} u(x-h,t)       }{h}.
	\end{eqnarray*}
	
	Rearranging, the equations using the implicit finite difference scheme we arrive at linear system with a penta-diagonal design matrix $\bM$ to be solved at each time step. In order to estimate the $\theta$, we use Proposition \ref{prop:bound} and estimate the bound on condition number using Gershgorin circle theorem which gives us $\operatorname{cond}(\bM) \leq 1.68$. Consequently, for $5$-list parallel sampling we pick $\theta = \frac{8.4}{n}$. The optimization process is investigated in Figure \ref{fig:penta}, and confirms that the convergence is super-linear in the later stages of the optimization (cca. $t>2^4$), and that with the increasing number of cores, the speedup achieved is nearly linear.

		\begin{figure}
			\centering
			\includegraphics[width=0.80\textwidth]{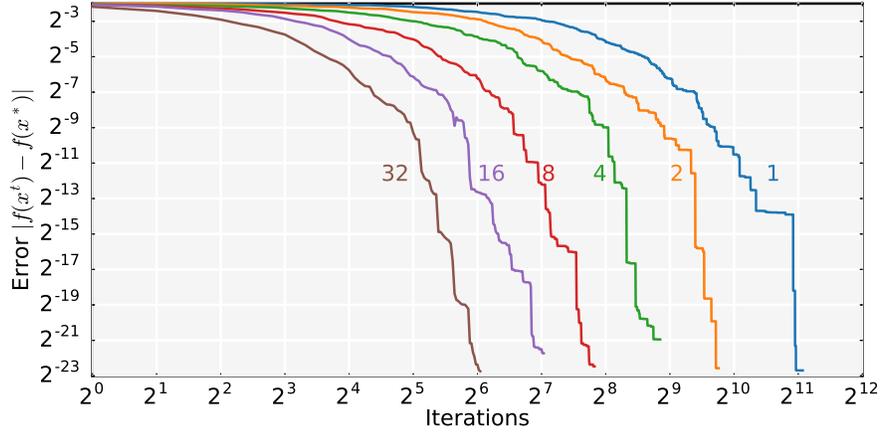}
			\caption{The number of iterations required to solve the linear system ($n = 10^3$) from the implicit finite difference scheme to the accuracy $10^{-8}$. The labels in the graph represents the number of computational units. }
			\label{fig:penta}
		\end{figure}

	\section{Empirical Risk Minimization in Parallel Settings}
	A specific application of this optimization method is optimizing the error function of statistical estimation called Empirical Risk Minimization (ERM). This was the main application area of the algorithm in \cite{SDNA} and \cite{Shalev-Shwartz2013} among many others. We reproduce for the sake of readers convenience the modifications to ERM formulations needed such that PSN can be directly applied.
	
	Many empirical risk minimization (ERM) problems can be cast as minimization of 
	\begin{eqnarray}\label{eq:primal}
		\min_{w \in \mathbb{R}^d} \left[ P(w) \eqdef \frac{1}{n} \sum_{i=1}^{n} \phi_i(a_i^\top w) + \lambda g(w) \right].
	\end{eqnarray}
	We assume $g$ is $1-$strongly convex with respect to $l_2$ norm and $\phi_i: \mathbb{R} \rightarrow \mathbb{R}$ are $\gamma_i$-strongly convex and smooth. The vector $a_i$ represents a feature vector of data point $i$.
	
	Using Fenchel duality theory, we are able to derive a dual optimization problem to the one in \eqref{eq:primal}. Fenchel conjugate function of $g$ is denoted $g^*$ and defined via $g^*(s) \eqdef \sup_{w\in \mathbb{R}^d} \braket{w,s} - g(w)$ in this work. Given this definition, we are able to formulate the dual problem where the solution is equivalent to the primal problem given strong duality holds.
	
	\begin{eqnarray} \label{eq:dual}
		\max_{\alpha \in \mathbb{R}^n} \left[   D(\alpha) \eqdef \frac{1}{n} \sum_{i=1}^{n} -\phi_i^*(-\alpha_i) - \lambda g^*\left(\frac{1}{n\lambda}\mathbf{A}\alpha\right) \right]
	\end{eqnarray}
	
	Also, we want to mention that under strong duality conditions the relation between primal and dual variables can be simply expressed as
	\begin{equation}
		w = w(\alpha ) = \nabla f (\mathbf{A}\lambda),
	\end{equation}
	and consequently $w(\alpha^*) = w^*$ where the star denotes the optimal solution to the optimization problems. 
	
	Redefining the functions in expression \eqref{eq:dual} by choosing $f(\alpha) = \lambda g^* (\frac{1}{\lambda n}\mathbf{A}\lambda)$ and $\psi_i (\alpha_i)= \frac{1}{n} \phi_i^*(-\alpha_i)$, and exchanging maximization for minimization, we yield the following problem
	
	\begin{equation}
		\min_{\alpha \in \mathbb{R}^n} \left[ F(\alpha) \eqdef f(\alpha) + \sum_{i=1}^{n} \psi_i(\alpha_i)\right],
	\end{equation}
	where $f$ satisfies the Assumption \ref{ass:smooth} and $\psi_i$ is strongly convex and smooth with the constant $\frac{1}{\gamma_i n}$. Thus, we see that the assumptions needed to apply Algorithm \ref{alg:Method2} are satisfied and we can define a specialized Algorithm \ref{alg:Method2} for this particular problem. 
	
	 Due to duality theory, $g$ being $1$-strongly convex implies that $g^*$ has $1$-Lipschitz continuous gradient \cite{Duenner2016}, and thus $\nabla^2 g^*(x) \preceq \mathbf{I}_{d\times d}$. Consequently, this implies that for $f(\alpha)$ defined via $g^*$ we have $\nabla^2 f(\alpha) \preceq \frac{1}{\lambda n^2} \mathbf{A}^\top\mathbf{A}$. Therefore, in this case, summing the two matrices to obtrain matrix for Assumption \ref{ass:smooth}, we get $\bX \eqdef \frac{1}{\lambda n^2} \mathbf{A}^\top\mathbf{A} +  \frac{\bD(\gamma)^{-1}}{n} $.
	
	\begin{algorithm}
		\caption{PSN: Parallel Stochastic Netwon for ERM}
		\label{alg:Method2}
		\renewcommand{\algorithmicrequire}{\textbf{Parameters:}}
		\renewcommand{\algorithmicensure}{\textbf{Initialization:}}
		\begin{algorithmic}[1]
			\Require Parallel Sampling $\hat{S}$, aggregation parameter $b$. 
			\Ensure Pick $\alpha_0 \in \mathbb{R}^n$ and $\bar{\alpha}_0 = \frac{1}{\lambda n } \mathbf{A}$
			\For{ $k = 0,1,2, \dots $ }
			\State Primal update: $w^k = \nabla g^*(\bar{\alpha}_k)$
			\State Generate a random set of sets $S_{k}$ distributed according to $\hat{S}$
			\For{ $j = 0, \dots, c $}
			\State $h^{k_j} \leftarrow \arg\min _{h \in \mathbb{R}^n} \braket{I_{S^k_j}\mathbf{A}^\top w^k,h} + \frac{1}{2}\braket{h,\mathbf{X}_{S^k_j}h} + \sum_{i \in S^k_j} \nabla \psi_i(\alpha^k_i)h_i$
			\EndFor
			\State Dual update:    $\alpha^{k+1} \leftarrow \alpha^k + \frac{1}{b}\sum_{j=1}^{c} h^{k_j}$ 
			\State Average update: $\bar{\alpha}^{k+1} \leftarrow \alpha^k + \frac{1}{n\lambda b} \sum_{j=1}^{c} \bA^\top h^{k_j}$
			\EndFor
			
		\end{algorithmic}
	\end{algorithm}
	As we have just applied the previous Algorithm \ref{alg:Method2} to solve a specific problem in \eqref{eq:dual}, the convergence rates are the same, where we just need to replace $\bM$ with $\bX$ in the Theorem \ref{theorem:Parallel_Method_1}. We do not present a bound on the duality gap due to technical difficulties, but we conjecture that the bound presented in \cite{SDNA} most likely holds.

%
%

	\section{Conclusion}
	We presented a novel way of parallelizing an existing algorithm called stochastic Newton (SN) introduced in \cite{SDNA}, which utilizes curvature information in data, should it be applied to statistical estimation problems, to improve on previous optimization methods. We prove that converge guarantees can be matched or improved over the serial version of the algorithm. The algorithm performs better than its coordinate counterpart parallel coordinate descent method (PCDM) both in theory and in practice. We demonstrated cases when the parallel version enjoys theoretical speedup over serial in special cases.
	\hfil \\ 
\noindent {\bf Acknowledgments.} The first author would like to thank EPRSC Vacational scholarship of the University of Edinburgh for supporting this work. The second author would like acknowledge support through EPSRC Early Career Fellowship in Mathematical Sciences.

\bibliography{bib.bib}   

\newpage	
\appendix
\section{Proof of Theorem \ref{theorem:Parallel_Method_1}}
	To handle the tedious expressions in the following theorem, we introduce the following notational shorthands.
	\begin{equation}\label{eq:x}
	\bX:=\mathbb{E}[(\mathbf{M}_{\hat{S}})^{-1}]^{1/2}\mathbf{G}\mathbb{E}[(\mathbf{M}_{\hat{S}})^{-1}]^{1/2}
	\end{equation}
	\begin{equation} \label{eq:defk}
	g^k:= \nabla f(x^k)
	\end{equation}
	\begin{equation} \label{eq:defz}
	z^k:= (\mathbb{E}[(\mathbf{M}_{\hat{S}})^{-1}])^{1/2}\nabla f(x^k)
	\end{equation}
	\begin{equation} \label{eq:y}
	\bY:=\mathbb{E}[(\mathbf{M}_{\hat{S}})^{-1}]^{1/2} \mathbf{M}  \mathbb{E}[(\mathbf{M}_{\hat{S}})^{-1}]^{1/2}
	\end{equation}
	
Recalling \eqref{eq:smooth} we get
		\begin{align}
			f(x^{k+1}) = f(x^k+h^k) \stackrel{\eqref{eq:smooth}} \leq f(x^k) + \braket{g^k,h^k} + \frac{1}{2}\braket{h^k,\mathbf{M}h^k}
		\end{align}
		Thus,
		\begin{eqnarray}
			f(x^{k+1})-f(x^k)  & \leq &  \braket{g^k,h^k} + \frac{1}{2}\braket{h^k,\mathbf{M}h^k} \\
			&\stackrel{\eqref{eq:update}}  =  &  -\frac{1}{b}\braket{g^k,\sum_{i=1}^{c}(\mathbf{M}_{S^k_i})^{-1} g^k} + \frac{1}{2b^2}\braket{\mathbf{M}\sum_{j=1}^{c}(\mathbf{M}_{S^k_j})^{-1} g^k,\sum_{i=1}^{c}(\mathbf{M}_{S^k_i})^{-1} g^k}  \\
			&\stackrel{\eqref{eq:slice}}  =  &  -\frac{1}{b}\braket{g^k,\sum_{i=1}^{c}(\mathbf{M}_{S^k_i})^{-1} g^k} + \frac{1}{2b^2}\braket{\mathbf{M}\sum_{j=1}^{c}\mathbf{I}_{S^k_j}(\mathbf{M}_{S^k_j})^{-1} g^k,\sum_{i=1}^{c}\mathbf{I}_{S^k_i}(\mathbf{M}_{S^k_i})^{-1} g^k} \notag  \\
			& = &  -\frac{1}{b}\sum_{i=1}^{c}\braket{g^k,(\mathbf{M}_{S^k_i})^{-1} g^k} + \frac{1}{2b^2}\sum_{j=1}^{c}\sum_{i=1}^{c}\braket{\mathbf{M}\mathbf{I}_{S^k_j}(\mathbf{M}_{S^k_j})^{-1} g^k,\mathbf{I}_{S^k_i}(\mathbf{M}_{S^k_i})^{-1} g^k}  \notag
		\end{eqnarray}
		We split the last term
		\begin{eqnarray}
			f(x^{k+1})-f(x^k)  & \leq & -\frac{1}{b}\sum_{i=1}^{c}\braket{g^k,(\mathbf{M}_{S_{k_i}})^{-1} g^k} + \frac{1}{2b^2}\sum_{i=1}^{c}\braket{\mathbf{M}\mathbf{I}_{S_{k_i}}(\mathbf{M}_{S_{k_i}})^{-1}g^k,\mathbf{I}_{S_{k_i}}(\mathbf{M}_{S_{k_i}})^{-1} g^k} \notag\\
			&& + \frac{1}{2b^2}\sum_{j\neq i}^{c}\braket{\mathbf{M}\mathbf{I}_{S^k_j}(\mathbf{M}_{S^k_j})^{-1}g^k,\mathbf{I}_{S^k_i}(\mathbf{M}_{S^k_i})^{-1} g^k}\notag\\
			& \stackrel{\eqref{eq:inverse},\eqref{eq:slice}}  = & -\frac{1}{b}\sum_{i=1}^{c}\braket{g^k,(\mathbf{M}_{S^k_i})^{-1} g^k} + \frac{1}{2b^2}\sum_{i=1}^{c}\braket{g^k,(\mathbf{M}_{S^k_i})^{-1} g^k} \\
			&& + \frac{1}{2b^2}\sum_{j\neq i}^{c}\braket{\mathbf{M}(\mathbf{M}_{S^k_j})^{-1}g^k,(\mathbf{M}_{S^k_i})^{-1} g^k}
		\end{eqnarray}
		\begin{eqnarray}
			\mathbb{E}[f(x^{k+1})-f(x^k)] &\stackrel{\eqref{eq:tower_prop}}\leq & -\left(\frac{1}{b}-\frac{1}{2b^2}\right)\sum_{i=1}^{c}\braket{g^k,\mathbb{E}[(\mathbf{M}_{S})^{-1}
				] g^k} \\ && + \frac{1}{2b^2}\sum_{j\neq i}^{c}\braket{\mathbf{M}\mathbb{E}[(\mathbf{M}_{\hat{S}})^{-1}]g^k,\mathbb{E}[(\mathbf{M}_{\hat{S}})^{-1}
				] g^k} \\
			& \stackrel{\eqref{eq:defz}}= & -\left(\frac{c}{b}-\frac{c}{2b^2}\right)\braket{z^k,z^k}  + \frac{c^2-c}{2b^2}\braket{\mathbf{M}\mathbb{E}[(\mathbf{M}_{\hat{S}})^{-1}]^{1/2} z^k,\mathbb{E}[(\mathbf{M}_{\hat{S}})^{-1}
				]^{1/2} z^k} \notag \\ & \stackrel{\eqref{eq:y}} =  &-\left(\frac{c}{b}-\frac{c}{2b^2}\right)\braket{z^k,z^k}  + \frac{c^2-c}{2b^2}\braket{ z^k,\mathbf{Y} z^k} 
			\\ & \stackrel{\eqref{eq:x},\eqref{eq:lambda}} \leq &-\left(\frac{c}{b}-\frac{c}{2b^2}\right)\braket{z^k,z^k}  + \frac{\lambda(c^2-c)}{2b^2}\braket{ z^k,\mathbf{X} z^k} 
			\\ & \stackrel{\eqref{eq:a_1}} \leq &-\left(\frac{c}{b}-\frac{c}{2b^2}\right)\braket{z^k,z^k}  + \frac{\lambda \theta (c^2-c)}{2b^2}\braket{ z^k, z^k}
			\\ & \stackrel{\eqref{eq:b}} \leq &-\frac{c}{2b}\braket{ z^k, z^k} \label{eq:neededforproof}
			\\ & \stackrel{\eqref{eq:defz},\eqref{eq:sigma_1}} \leq &-\frac{c}{2b} \sigma_1 \braket{ g^k, \mathbf{G}^{-1} g^k}
			\\ & \stackrel{\eqref{eq:strgcnvx2},\eqref{eq:sigma_p}}\leq & -\sigma_p(f(x^k)-f(x^*))
		\end{eqnarray}

	\begin{remark}
		When $c=1$ we get back to case described by Theorem \ref{theorem:Serial_Method_1}.	
	\end{remark}

\end{document}